\documentclass[reqno]{amsart}

\usepackage{amsthm}
\usepackage[usenames,dvipsnames]{pstricks}
\usepackage{epsfig}
\usepackage{pst-grad} 
\usepackage{pst-plot} 
\usepackage{graphicx,subfigure}
\usepackage{amsmath,amssymb}

\usepackage{acronym}
\acrodef{BO}{{\sl Benjamin-Ono}}
\acrodef{rBO}{{\sl regularized Benjamin-Ono}}
\acrodef{rILW}{{\sl regularized Intermediate Long Wave}}
\acrodef{DSW}{{\sl Dispersive Shock Wave}}
\acrodef{DSWs}{{\sl Dispersive Shock Waves}}
\acrodef{ILW}{{\sl Intermediate Long Wave}}
\acrodef{CGN}{{\sl Conjugate Gradient-Newton}}
\acrodef{SW/SW}{{\sl Shallow water / Shallow water}}
\acrodef{B/B}{{\sl Boussinesq / Boussinesq}}

\makeatletter
\newcommand{\sech}{\mathop{\operator@font sech}}
\newcommand{\sign}{\mathop{\operator@font sign}}
\makeatother

\newtheorem{theorem}{Theorem}[section]

\newtheorem{corollary}{Corollary}[section]

\numberwithin{equation}{section}

\begin{document}

\title[]{Asymptotic decay of solitary wave solutions of the fractional nonlinear Schr\"{o}dinger equation}


\author{Angel Dur\'an}
\address{\textbf{A.~Dur\'an:} Applied Mathematics Department, University of Valladolid, P/ Belen 15, 47011, Valladolid, Spain}
\email{angel@mac.uva.es}

\author{Nuria Reguera}
\address{\textbf{N.~Reguera:} Department of Mathematics and Computation, University of Burgos, 09001 Burgos, Spain}
\email{nreguera@ubu.es}



\subjclass[2010]{76B25,35C07,65H10}



\keywords{Fractional nonlinear Schr\"{o}dinger equations, solitary waves, asymptotic decay}

\begin{abstract}
The existence of solitary wave solutions of the one-dimensional version of the fractional nonlinear Schr\"{o}dinger (fNLS) equation was analyzed by the authors in a previous work. In this paper, the asymptotic decay of the solitary waves is analyzed. From the formulation of the differential system for the wave profiles as a convolution, these are shown to decay algebraically to zero at infinity, with an order which depends on the parameter determining the fractional order of the equation. Some numerical experiments illustrate the result.
\end{abstract}

\maketitle

\section{Introduction}\label{sec1}
In a previous paper, \cite{DR1}, the present authors analyzed the existence of solitary wave solutions of the 1D fractional nonlinear Schr\"{o}dinger (fNLS) equation
\begin{eqnarray}
iu_{t}- (-\partial_{xx})^{s}u+ |u|^{2\sigma}u=0,\quad {x}\in \mathbb{R},\quad t>0,\label{fnls1d}
\end{eqnarray} 
with $\sigma>0, 0<s<1$, and the nonlocal operator $(-\partial_{xx})^{s}$ is of Fourier multiplier type with
\begin{eqnarray*}
\widehat{(-\partial_{xx})^{s}f}(\xi)=|\xi|^{2s}\widehat{f}(\xi),\quad \xi\in\mathbb{R}^{n},\label{fnls2}
\end{eqnarray*}
where
$
\widehat{f}(\xi)$
denotes the Fourier transform of $f$ at $\xi$. Equation (\ref{fnls1d}) was originally introduced in the area of quantum mechanics from, e.~g., the papers by Laskin, \cite{Laskin2000,Laskin2002,Laskin2011}, and Fr\"{o}lich et al., \cite{FrohlichJL2007}.

In order to study the existence of solitary wave solutions of (\ref{fnls1d}), this is written as a real system
\begin{eqnarray}
v_{t}-(-\partial_{xx})^{s}w+(v^{2}+w^{2})^{\sigma}w&=&0,\nonumber\\
-w_{t}-(-\partial_{xx})^{s}v+(v^{2}+w^{2})^{\sigma}v&=&0.\label{fnls1b}
\end{eqnarray}
for $u=v+iw$, $v$ and $w$ real functions of $(x,t)$. As in the case of the classical nonlinear Schr\"{o}dinger equation, correspondimg to $s=1$ in (\ref{fnls1d}), the formation of solitary wave solutions in (\ref{fnls1b}) is based on the conserved quantities
\begin{eqnarray}
I_{1}(v,w)&=&\frac{1}{2}\int_{\mathbb{R}}(v^{2}+w^{2})dx=\frac{1}{2}\int_{\mathbb{R}}|u|^{2}dx, \label{fnls3a}\\
I_{2}(v,w)&=&\frac{1}{2}\int_{\mathbb{R}}(vw_{x}-wv_{x})dx=\frac{1}{2}\int_{\mathbb{R}}{\rm Im}(u\overline{u}_{x})dx, \label{fnls3b}\\
H(v,w)&=&\int_{\mathbb{R}}\left(\frac{1}{2}\left( (|D|^{s}v)^{2}+(|D|^{s}w)^{2}\right)-\frac{1}{2\sigma+2} (v^{2}+w^{2})^{\sigma+1}\right) d{x},\label{fnls3c}
\end{eqnarray}
where $|D|^{s}=(-\partial_{xx})^{s/2}$. The invariants (\ref{fnls3a}), (\ref{fnls3b}) determine, respectively, the infinitesimal generators of the symmetry group of (\ref{fnls1b}), \cite{Olver}, consisting of rotations and spatial translations
$(v,w)\mapsto G_{(\alpha,\beta)}(v,w), \alpha,\beta\in\mathbb{R}$ where
\begin{eqnarray}
G_{(\alpha,\beta)}(v,w)(x)=\begin{pmatrix} \cos\alpha&-\sin\alpha\\\sin\alpha&\cos\alpha\end{pmatrix}\begin{pmatrix}v(x-\beta)\\w(x-\beta)\end{pmatrix},\label{symg}
\end{eqnarray}
in the sense that if $(v(x,t),w(x,t))$ is a solution of (\ref{fnls1b}), then $(\widetilde{v},\widetilde{w})= G_{(\alpha,\beta)}(v,w)$ is a solution of (\ref{fnls1b}). On the other hand, (\ref{fnls3c}) is the energy function of the Hamiltonian structure of (\ref{fnls1b})
\begin{eqnarray*}
\frac{d}{dt}\begin{pmatrix}v\\w\end{pmatrix}{\mathcal J}\delta H(v,w),\quad \mathcal{J}=\begin{pmatrix}0&1\\-1&0\end{pmatrix},
\end{eqnarray*}
where $\delta H$ denotes variational (Fr\'echet) derivative of (\ref{fnls3c}).

Within this framework, the solitary wave solutions of (\ref{fnls1b}) are obtained from critical points $(v_{0},w_{0})$ of the Hamiltonian (\ref{fnls3c}) constrained to fixed values of the quantities (\ref{fnls3a}), (\ref{fnls3b}), that is
\begin{eqnarray}
\delta\left(H(u_{0})-\lambda_{0}^{1}I_{1}(u_{0})-\lambda_{0}^{2}I_{2}(u_{0})\right)=0,\label{fnlsCCP}
\end{eqnarray}
for real Lagrange multipliers $\lambda_{0}^{j}, j=1,2$. The constrained critical point condition (\ref{fnlsCCP}) takes the form of a nonlocal differential system for $(v_{0},w_{0})$
\begin{eqnarray}
-(-\partial_{xx})^{s}v_{0}-\lambda_{0}^{1}v_{0}+\lambda_{0}^{2}w_{0}'+(v_{0}^{2}+w_{0}^{2})^{\sigma}v_{0}&=&0,\nonumber\\
-(-\partial_{xx})^{s}w_{0}-\lambda_{0}^{1}w_{0}-\lambda_{0}^{2}v_{0}'+(v_{0}^{2}+w_{0}^{2})^{\sigma}w_{0}&=&0.\label{fnls22_2}
\end{eqnarray}
If $u_{0}(x)=v_{0}(x)+iw_{0}(x)=\rho(x)e^{i\theta(x)}$, with real $\rho,\theta$, a two-parameter family of solitary wave solutions of (\ref{fnls1d}) is obtained from the action of the symmetry group (\ref{symg}), cf. \cite{DR1,DuranS2000}
\begin{eqnarray}
\psi(x,t,a,c,x_{0},\theta_{0})=G_{(\theta_{0}+t\lambda_{0}^{1},x_{0}+t\lambda_{0}^{2})}(\varphi)=\rho(x-t\lambda_{0}^{2}-x_{0})e^{i(\theta(x-t\lambda_{0}^{2}-x_{0})+\theta_{0}+\lambda_{0}^{1}t)},\label{fnls22_7}
\end{eqnarray}
where $\theta_{0},x_{0}\in\mathbb{R}$. The existence of solutions $(v_{0},w_{0})\in H^{\infty}\times H^{\infty}$ of (\ref{fnls22_2}), for 
\begin{eqnarray}
s\in (1/2,1),\quad  \lambda_{0}^{1}>0, \quad |\lambda_{0}^{2}|<c(\lambda_{0}^{1})=2s\left(\frac{\lambda_{0}^{1}}{2s-1}\right)^{\frac{2s-1}{2s}},\label{fnls_236b}
\end{eqnarray}
is proved in \cite{DR1}, from the application of the Concentration-Compactness theory of Lions, \cite{Lions} to certain constrained minimization problems. In addition, for the particular subfamily of (\ref{fnls22_7}) with $\theta(x)=Ax$, where, \cite{HongS2015,HongS2017}
\begin{eqnarray*}
\lambda_{0}^{2}=2s|A|^{2s-2}A,
\end{eqnarray*}
it is shown that
\begin{eqnarray}
|x|^{2s+1}\rho(x)\leq C,\quad x\in\mathbb{R},\label{fnls_2328}
\end{eqnarray}
for some constant $C>0$.

The main purpose of the present note is to generalize the decay estimate (\ref{fnls_2328}) to all the solutions of (\ref{fnls22_2}) under the conditions (\ref{fnls_236b}). To this end, the system (\ref{fnls22_2}) will be written in a convolution form in order to apply the theory developed in \cite{BonaL,ChenB}. This requires a detailed study of the kernels of the convolution and the adaptation of the theory to the system under study.

The paper is structured as follows. In section \ref{sec2}, the convolution problem is introduced and the regularity and decay of the solutions are analyzed. The decay of the components of the matrix kernel is the key to the proof of the main result on decay estimates of the solutions of trhe convolution system and therefore of the solitary wave solutions of 
(\ref{fnls22_2}). This is developed in section \ref{sec3}. Some numerical experiments to illustrate the result are given in section \ref{sec4}.

The following notation will be used throughout the paper. For $1\leq p\leq\infty$, the norm of the $L^{p}=L^{p}(\mathbb{R})$ space will be denoted as $||\cdot ||_{L^{p}}$. For $\mu\geq 0$, $H^{\mu}=H^{\mu}(\mathbb{R})$ will stand for the $L^{2}$-based Sobolev space of order $\mu$, with $H^{0}=L^{2}$ and associated norm denoted by $||\cdot ||_{H^{\mu}}$. In addition, for $u=(v,w)$ we define
$$|u|_{2}=(v^{2}+w^{2})^{1/2}.$$
\section{The convolution system and regularity}
\label{sec2}
\subsection{The convolution problem}
Let $Q$ be the matrix operator with Fourier symbol
\begin{eqnarray}
\widehat{Q}(\xi)=\begin{pmatrix}\lambda_{0}^{1}+|\xi|^{2s}&-i\lambda_{0}^{2}\xi\\i\lambda_{0}^{2}\xi &\lambda_{0}^{1}+|\xi|^{2s}\end{pmatrix},\quad \xi\in\mathbb{R}.\label{fnls_235}
\end{eqnarray}
According to Lemma 2.4 of \cite{DR1}, under the conditions (\ref{fnls_236b}), the operator $Q$ is definite positive. Then we can consider the convolution system of the form
\begin{eqnarray}
\begin{pmatrix}v\\w\end{pmatrix}=\mathcal{K}\ast G(v,w),\label{fnlsC1}
\end{eqnarray}
where $\mathcal{K}=(k_{ij})_{i,j)1}^{2}=Q^{-1}$ and therefore with Fourier representation
\begin{eqnarray*}
\widehat{\mathcal{K}}(\xi)=(\widehat{k}_{ij})_{i,j=1}^{2}=\frac{1}{\lambda(\xi)}\widehat{Q}(\xi)^{T},\label{fnlsC2}
\end{eqnarray*}
with $\lambda(\xi)=\lambda_{+}(\xi)\lambda_{-}(\xi)$,
\begin{eqnarray}
\lambda_{\pm}(\xi)=|\xi|^{2s}+\lambda_{0}^{1}\pm\lambda_{0}^{2}\xi,\label{fnls_237}
\end{eqnarray}
and
\begin{eqnarray}
\widehat{k}_{11}(\xi)=\widehat{k}_{22}(\xi)=\frac{\lambda_{0}^{1}+|\xi|^{2s}}{\lambda(\xi)},\quad 
\widehat{k}_{12}(\xi)=\overline{\widehat{k}_{21}(\xi)}=\frac{i\xi \lambda_{0}^{2}}{\lambda(\xi)},\quad \xi\in \mathbb{R}.\label{fnlsC3}
\end{eqnarray}
In addition, the nonlinear term $G$ in (\ref{fnlsC1}) is given by
\begin{eqnarray}
G(v,w)=\begin{pmatrix} g_{1}(v,w)\\g_{2}(v,w)\end{pmatrix}=(v^{2}+w^{2})^{\sigma}\begin{pmatrix}v\\w\end{pmatrix},\label{fnlsC3b}
\end{eqnarray}
so the system (\ref{fnlsC1}) is written as
\begin{eqnarray}
v&=&k_{11}\ast g_{1}(v,w)+k_{12}\ast g_{2}(v,w)\nonumber\\
&=&\int_{\mathbb{R}}\left(k_{11}(x-y)g_{1}(v(y),w(y))+k_{12}(x-y)g_{2}(v(y),w(y))\right)dy,\nonumber\\
w&=&k_{21}\ast g_{1}(v,w)+k_{22}\ast g_{2}(v,w)\nonumber\\
&=&\int_{\mathbb{R}}\left(k_{21}(x-y)g_{1}(v(y),w(y))+k_{22}(x-y)g_{2}(v(y),w(y))\right)dy.\label{fnlsC4}
\end{eqnarray}
Note that the representation of (\ref{fnls22_2}) in Fourier space is given by
\begin{eqnarray}
\widehat{Q}(\xi)\begin{pmatrix} \widehat{v}_{0}(\xi)\\\widehat{w}_{0}(\xi)\end{pmatrix}=\widehat{G(v_{0},w_{0})}(\xi),\quad \xi\mathbb{R}.\label{fnlsC5}
\end{eqnarray}
Therefore, taking the inverse Fourier transform in (\ref{fnlsC5}), then a solution $(v_{0},w_{0})\in H^{s}\times H^{s}, 1/2<s<1$ of (\ref{fnls22_2}) is also a solution of (\ref{fnlsC5}).
\subsection{Asymptotic decay of the kernels}
In this section some decay properties of the kernels $k_{ij}$ will be studied. These are necessary in order to analyze the asymptotic decay of the solutions of (\ref{fnlsC4}) in section \ref{sec3}.
\begin{theorem}
\label{theor0}
Let us assume that (\ref{fnls_236b}) holds and let $\mathcal{K}=(k_{ij})_{i,j)1}^{2}$ be defined 
in (\ref{fnlsC1}) with Fourier transforms of the kernels given in (\ref{fnlsC3}). Then
\begin{eqnarray}
\lim_{|x|\rightarrow\infty}|x|^{2s+1}k_{11}(x)&=&K_{1}:=\frac{\sin{s\pi}}{\pi}\int_{0}^{\infty}z^{2s}e^{-z}dz.\label{fnlsC7a}\\
\lim_{|x|\rightarrow\infty}|x|^{2s+2}k_{12}(x)&=&K_{2}:=-\frac{2\lambda_{0}^{2}\sin{s\pi}}{\pi}\int_{0}^{\infty}z^{2s+1}e^{-z}dz.\label{fnlsC7b}
\end{eqnarray}
\end{theorem}
\begin{proof}
Note, from (\ref{fnlsC3}) and since $s\in (1/2,1)$, that for $i,j=1,2$, $k_{ij}$ is a measurable function with $\widehat{k}_{ij}\in H^{s}$. In addition
\begin{eqnarray*}
\widehat{k}_{jj}(-\xi)=\widehat{k}_{jj}(-\xi),\; j=1,2,\quad \widehat{k}_{ij}(-\xi)=-\widehat{k}_{ij}(-\xi),\; i,j=1,2, i\neq j,\quad \xi\in\mathbb{R}.
\end{eqnarray*}
Therefore, from the inverse Fourier transform we have
\begin{eqnarray}
k_{11}(x)&=&k_{22}(x)=\frac{1}{\pi}\int_{0}^{\infty}\frac{\lambda_{0}^{1}+|\xi|^{2s}}{\lambda(\xi)}\cos{\xi x}d\xi,\label{fnlsC6a}\\
k_{12}(x)&=&-k_{21}(x)=-\frac{1}{\pi}\int_{0}^{\infty}\frac{\xi \lambda_{0}^{2}}{\lambda(\xi)}\sin{\xi x}d\xi.\label{fnlsC6b}
\end{eqnarray}
Using the techniques in, e.~g. \cite{ChenB}, we evaluate the integrals in (\ref{fnlsC6a}), (\ref{fnlsC6b}). In the case of (\ref{fnlsC6a}), let
\begin{eqnarray*}
f(z)=\frac{\lambda_{0}^{1}+z^{2s}}{\widetilde{\lambda}(z)}e^{i\omega z},\quad \omega\geq 0,\quad
\widetilde{\lambda}(z)=\widetilde{\lambda_{+}}(z)\widetilde{\lambda_{-}}(z),\quad
\widetilde{\lambda_{\pm}}(z)=z^{2s}+\lambda_{0}^{1}\pm\lambda_{0}^{2}z,
\end{eqnarray*}
where we consider the branch of the logarithm which makes $1^{2s}=1$, and assume that (\ref{fnls_236b}) holds, with $\lambda_{0}^{2}>0$ for simplicity. From (\ref{fnls_237}), it is not hard to check that $\widetilde{\lambda}(z)$ cannot have multiple zeros: Indeed, we write
\begin{eqnarray*}
\widetilde{\lambda}(z)=(\lambda_{0}^{1}+z^{2s})^{2}-(\lambda_{0}^{2})^{2}z^{2}.
\end{eqnarray*}
Let $z_{0}$ be such that $\widetilde{\lambda}(z_{0})=0$. Then $\widetilde{\lambda_{+}}(z_{0})=0$ or $\widetilde{\lambda_{-}}(z_{0})=0$, that is
\begin{eqnarray*}
(\lambda_{0}^{1}+z^{2s})=-\lambda_{0}^{2}z_{0}\quad {\rm or}\quad (\lambda_{0}^{1}+z^{2s})=\lambda_{0}^{2}z_{0}.
\end{eqnarray*}
Assume that $\widetilde{\lambda_{-}}(z_{0})=0$. Then
\begin{eqnarray*}
\widetilde{\lambda}'(z_{0})=4s\lambda_{0}^{2}(\lambda_{0}^{2}z_{0}-\lambda_{0}^{1})-2z_{0}(\lambda_{0}^{2})^{2}.
\end{eqnarray*}
Therefore, $\widetilde{\lambda}'(z_{0})=0$ implies 
\begin{eqnarray*}
z_{0}=\frac{2s}{2s-1}\frac{\lambda_{0}^{1}}{\lambda_{0}^{2}},
\end{eqnarray*}
and the condition $\widetilde{\lambda_{-}}(z_{0})=0$ reads
\begin{eqnarray*}
\lambda_{0}^{1}+\left(\frac{2s}{2s-1}\frac{\lambda_{0}^{1}}{\lambda_{0}^{2}}\right)^{2s}=\frac{2s}{2s-1}{\lambda_{0}^{1}}.
\end{eqnarray*}
Thus
\begin{eqnarray*}
\frac{\lambda_{0}^{2}}{2s}=\left(\frac{\lambda_{0}^{1}}{2s-1}\right)^{\frac{2s-1}{2s}},
\end{eqnarray*}
which contradicts (\ref{fnls_236b}). If $\widetilde{\lambda_{+}}(z_{0})=0$, then the condition $\widetilde{\lambda}'(z_{0})=0$ implies
\begin{eqnarray*}
z_{0}=-\frac{2s}{2s-1}\frac{\lambda_{0}^{1}}{\lambda_{0}^{2}},
\end{eqnarray*}
and similar arguments also lead to contradiction with (\ref{fnls_236b}). Hence, $\widetilde{\lambda}(z)$ cannot have multiple zeros.

We define
\begin{eqnarray*}
\Omega=\{z=\rho e^{i\theta}, \epsilon\leq \rho\leq R, 0\leq \theta\leq \pi/2\},
\end{eqnarray*}
and from Rouche's theorem we may consider $\epsilon>0$ small enough and $R>0$ large enough such that there is no $z$ outside $\Omega$ and inside the first quadrant making $\widetilde{\lambda}(z)=0$; then $f(z)$ is analytic with finitely many simple poles 
$z_{j}=x_{j}+iy_{j}, y_{j}>0, j=1,\ldots,M$, in $\Omega$. We mow apply the Residue theorem and take limits $\epsilon\rightarrow 0$ and $R\rightarrow\infty$ to have
\begin{eqnarray*}
k_{11}(\omega)=k_{22}(\omega)=\frac{1}{\pi}\int_{0}^{\infty}f(x)dx
=\frac{1}{\pi}{\rm Re}\left(i\int_{0}^{\infty}f(iy)dy+2\pi i\sum_{j=1}^{M}{\rm Res}(f(z),z_{j})\right).
\end{eqnarray*}
After some computations, we have
\begin{eqnarray*}
I_{1}={\rm Re}\left(i\int_{0}^{\infty}f(iy)dy\right)=\int_{0}^{\infty}\frac{e^{-\omega y}}{A(z)^{2}+B(z)^{2}}\left(B(z)(\lambda_{0}^{1}+y^{2s}\cos{s\pi})-A(z)y^{2s}\sin{s\pi}\right)dy,
\end{eqnarray*}
where
\begin{eqnarray*}
A(z)&=&(\lambda_{0}^{1}+y^{2s}\cos{s\pi})^{2}+(\lambda_{0}^{2})^{2}y^{2}-y^{4s}(\sin{s\pi})^{2},\\
B(z)&=&2y^{2s}\sin{s\pi}(\lambda_{0}^{1}+y^{2s}\cos{s\pi}).
\end{eqnarray*}
The change of variables $z=\omega y$ leads to
\begin{eqnarray*}
I_{1}&=&\frac{1}{\omega^{2s+1}}\int_{0}^{\infty}G_{1}(z,\omega)dz\\
&=&\frac{1}{\omega^{2s+1}}\int_{0}^{\infty}\frac{e^{-z}}{\widetilde{A}(z)^{2}+\widetilde{B}(z)^{2}}\left(z^{2s}\sin{s\pi}\left(\left(\lambda_{0}^{1}+\frac{z^{2s}}{\omega^{2s}}\cos{s\pi}\right)^{2}-(\lambda_{0}^{2})^{2}\frac{z^{2}}{\omega^{2}}+\frac{z^{4s}}{\omega^{4s}}(\sin{s\pi})^{2}\right)\right)dz,
\end{eqnarray*}
with
\begin{eqnarray*}
\widetilde{A}(z)&=&(\lambda_{0}^{1}+\left(\frac{z}{\omega}\right)^{2s}\cos{s\pi})^{2}+(\lambda_{0}^{2})^{2}\left(\frac{z}{\omega}\right)^{2}-\left(\frac{z}{\omega}\right)^{4s}(\sin{s\pi})^{2},\\
\widetilde{B}(z)&=&2\left(\frac{z}{\omega}\right)^{2s}\sin{s\pi}(\lambda_{0}^{1}+\left(\frac{z}{\omega}\right)^{2s}\cos{s\pi}).
\end{eqnarray*}
Note that
\begin{eqnarray*}
\lim_{\omega\rightarrow\infty}\widetilde{A}(z)^{2}+\widetilde{B}(z)^{2}=(\lambda_{0}^{1})^{2}.
\end{eqnarray*}
Therefore
\begin{eqnarray*}
k_{11}(x)=\frac{1}{\pi |x|^{2s+1}}\int_{0}^{\infty}G_{1}(z,|x|)dz+{\rm Re}\sum_{j=1}^{M}C_{j}e^{ix_{j}|x|-y_{j}|x|},
\end{eqnarray*}
for some constants $C_{j}\in\mathbb{C}, j=1,\ldots,M$. Thus
\begin{eqnarray*}
\lim_{|x|\rightarrow\infty}|x|^{2s+1}k_{11}(x)=\frac{1}{\pi}\lim_{|x|\rightarrow\infty}\int_{0}^{\infty}G_{1}(z,|x|)dz=\frac{\sin{s\pi}}{\pi}\int_{0}^{\infty}z^{2s}e^{-z}dz,
\end{eqnarray*}
and (\ref{fnlsC7a}) holds.
For the case of $k_{12}(x)$, we take
\begin{eqnarray*}
g(z)=\frac{iz(\lambda_{0}^{2})e^{ix\omega}}{\widetilde{\lambda}(z)},\quad \omega\geq 0.
\end{eqnarray*}
Similar computations lead to
\begin{eqnarray*}
k_{12}(\omega)=-\frac{1}{\pi}{\rm Re}\left(i\int_{0}^{\infty}g(iy)dy+2\pi i\sum_{j=1}^{M}{\rm Res}(g(z),z_{j})\right),
\end{eqnarray*}
and
\begin{eqnarray*}
I_{2}&=&{\rm Re}\left(i\int_{0}^{\infty}g(iy)dy\right)=\int_{0}^{\infty}\frac{y\lambda_{0}^{2}e^{-\omega y}}{A(z)^{2}+B(z)^{2}}dy=\frac{1}{\omega^{2s+2}}\int_{0}^{\infty}G_{2}(z,\omega)dz\\
&=&\frac{1}{\omega^{2s+2}}\int_{0}^{\infty}frac{2\lambda_{0}^{2}z^{2s+1}e^{-z}\sin{s\pi}}{\widetilde{A}(z)^{2}+\widetilde{B}(z)^{2}}(\lambda_{0}^{1}+\left(\frac{z}{\omega}\right)^{2s}\cos{s\pi})^{2}dz.
\end{eqnarray*}
Therefore
\begin{eqnarray*}
k_{12}(x)=-\frac{1}{\pi |x|^{2s+2}}\int_{0}^{\infty}G_{2}(z,|x|)dz+{\rm Re}\sum_{j=1}^{M}C_{j}e^{ix_{j}|x|-y_{j}|x|},
\end{eqnarray*}
and
\begin{eqnarray*}
\lim_{|x|\rightarrow\infty}|x|^{2s+2}k_{12}(x)=-\frac{2\lambda_{0}^{2}\sin{s\pi}}{\pi}\int_{0}^{\infty}z^{2s+1}e^{-z}dz,
\end{eqnarray*}
which proves (\ref{fnlsC7b}).
\end{proof}
\section{Decay and regularity}
\label{sec3}
In order to study the decay and regularity of the solutions of (\ref{fnlsC4}), we will follow the approach developed in \cite{BonaL} and adapt it to our system. Note first that, if $u=(v,w)$, then, from (\ref{fnlsC3b}), the functions $g_{j}(v,w), j=1,2$, are measurable with
\begin{eqnarray}
|G(u)|_{2}=\left(g_{1}(v,w)^{2}+g_{2}(v,w)^{2}\right)^{1/2}=|u|_{2}^{r},\label{fnlsC8}
\end{eqnarray}
where $r=2\sigma+1>1$. In addition, recall that, as mentioned before, for $s\in (1/2,1)$, and from (\ref{fnlsC3}), $k_{ij}$ is a measurable function with $\widehat{k}_{ij}\in H^{s}, j=1,2$.
\begin{theorem}
\label{theor1} Let $(v,w)\in L^{\infty}\times L^{\infty}$ be a solution of (\ref{fnlsC4}) with
\begin{eqnarray}
\lim_{|x|\rightarrow\infty}v(x)=\lim_{|x|\rightarrow\infty}w(x)=0.\label{fnlsC9}
\end{eqnarray}
Then $v,w\in L^{1}\cap L^{2}$ and there is $l\in (0,s)$ such that $|x|^{l}v(x), |x|^{l}w(x)\in L^{2}\cap L^{\infty}$.
\end{theorem}
\begin{proof}
We adapt the arguments made in Theorem 3.1.2 of \cite{BonaL} to the system (\ref{fnlsC4}). Let $l>0$ be such that $s>l+1/2$, $0<\epsilon\leq 1$, and define
\begin{eqnarray*}
v_{\epsilon}(x)=\frac{|x|^{l}}{(1+\epsilon |x|)^{s}}v(x), \quad
w_{\epsilon}(x)=\frac{|x|^{l}}{(1+\epsilon |x|)^{s}}w(x).
\end{eqnarray*}
Since $v,w\in L^{\infty}$, then $v_{\epsilon}, w_{\epsilon}\in L^{2}$. In addition, since $r>1$ and from (\ref{fnlsC9}), then, for any $\delta>0$, there exists $N\geq 1$ such that
\begin{eqnarray}
|u(x)|_{2}^{r-1}\leq \delta,\quad u(x)=(v(x),w(x)),\label{fnlsC10}
\end{eqnarray}
for almost all $x$ with $|x|\geq N$. Note, using (\ref{fnlsC8}), that (\ref{fnlsC10}) implies
\begin{eqnarray}
|g_{j}(v,w)(x)|\leq \delta (v(x)^{2}+w(x)^{2})^{1/2},\quad j=1,2,\label{fnlsC11}
\end{eqnarray}
for almost all $x$ with $|x|\geq N$. Using the property
\begin{eqnarray*}
(1+|x-y|)^{2s}\leq (2(1+|x-y|^{2}))^{s}=2^{s}(1+|x-y|^{2})^{s},
\end{eqnarray*}
and defining
\begin{eqnarray*}
A_{11}&=&\left(\int_{\mathbb{R}}(1+|x-y|)^{2s}|k_{11}(x-y)|^{2}dy\right)^{1/2},\\
A_{12}&=&\left(\int_{\mathbb{R}}(1+|x-y|)^{2s}|k_{12}(x-y)|^{2}dy\right)^{1/2},
\end{eqnarray*}
then we have (cf. \cite{BonaL})
\begin{eqnarray}
\int_{N}^{\infty}|v_{\epsilon}(x)|^{2}dx
&=&\int_{N}^{\infty}|v_{\epsilon}(x)|\frac{|x|^{l}}{(1+\epsilon |x|)^{s}}|v(x)|dx\nonumber\\
&=&\int_{N}^{\infty}|v_{\epsilon}(x)|\frac{|x|^{l}}{(1+\epsilon |x|)^{s}}\left|\int_{\mathbb{R}}\left(k_{11}(x-y)g_{1}(v(y),w(y))\right.\right.\nonumber\\
&&\left.\left.+k_{12}(x-y)g_{2}(v(y),w(y))dy\right)dx\right|\nonumber\\
&\leq&\int_{N}^{\infty}|v_{\epsilon}(x)|\frac{|x|^{l}}{(1+\epsilon |x|)^{s}}\int_{\mathbb{R}}\left(|k_{11}(x-y)g_{1}(v(y),w(y))|\right.\nonumber\\
&&\left.+|k_{12}(x-y)g_{2}(v(y),w(y))|\right)dydx\nonumber\\
&\leq&\int_{N}^{\infty}|v_{\epsilon}(x)|\frac{|x|^{l}}{(1+\epsilon |x|)^{s}}A_{11}\left(\int_{\mathbb{R}}\frac{|g_{1}(v(y),w(y))|^{2}}{(1+|x-y|)^{2s}}dy\right)^{1/2}dx\nonumber\\
&&+\int_{N}^{\infty}|v_{\epsilon}(x)|\frac{|x|^{l}}{(1+\epsilon |x|)^{s}}A_{12}\left(\int_{\mathbb{R}}\frac{|g_{2}(v(y),w(y))|^{2}}{(1+|x-y|)^{2s}}dy\right)^{1/2}dx\label{fnlsC12}\\
&\leq&2^{s}||\widehat{k}_{11}||_{H^{s}}\left(\int_{N}^{\infty}|v_{\epsilon}(x)|^{2}dx\right)\left(\int_{N}^{\infty}\frac{|x|^{2l}}{(1+\epsilon |x|)^{2s}}\int_{\mathbb{R}}\frac{|g_{1}(v(y),w(y))|^{2}}{(1+|x-y|)^{2s}}dydx\right)^{1/2}\nonumber\\
&&+2^{s}||\widehat{k}_{12}||_{H^{s}}\left(\int_{N}^{\infty}|v_{\epsilon}(x)|^{2}dx\right)\left(\int_{N}^{\infty}\frac{|x|^{2l}}{(1+\epsilon |x|)^{2s}}\int_{\mathbb{R}}\frac{|g_{2}(v(y),w(y))|^{2}}{(1+|x-y|)^{2s}}dydx\right)^{1/2}.\nonumber
\end{eqnarray}
We now make use of Fubini's theorem, (\ref{fnlsC10}), (\ref{fnlsC11}), and the inequalities (cf. \cite{BonaL})
\begin{eqnarray*}
\int_{0}^{\infty}\frac{t^{2l}dt}{(1+\epsilon t )^{2s}(1+|x-t|)^{2s}}&\leq&\frac{B|x|^{2l}}{(1+\epsilon|x|)^{2s}},\\
\int_{N}^{\infty}\frac{t^{2l}dt}{(1+\epsilon t)^{2s}(1+|x-t|)^{2s}}&\leq&\int_{\mathbb{R}}\frac{(|x|+N)^{2l}dx}{(1+|x|)^{2s}},\quad t\in [-N,N],
\end{eqnarray*}
which hold for some constant $B$ and any $x, \epsilon$ with $|x|\geq 1, 0<\epsilon\leq 1$. We also define
\begin{eqnarray*}
R=\left(\int_{-\infty}^{\infty}\frac{(|x|+N)^{2l}dx}{(1+|x|)^{2s}}\right)^{1/2}.
\end{eqnarray*}
With all this, (\ref{fnlsC12}) becomes
\begin{eqnarray}
\left(\int_{N}^{\infty}|v_{\epsilon}(x)|^{2}dx\right)^{1/2}&\leq&2^{s}||\widehat{k}_{11}||_{H^{s}}\left(\int_{\mathbb{R}}dy{|g_{1}(v(y),w(y))|^{2}}\int_{N}^{\infty}\frac{|x|^{2l}}{(1+\epsilon |x|)^{2s}(1+|x-y|)^{2s}}\right)^{1/2}\nonumber\\
&&+2^{s}||\widehat{k}_{12}||_{H^{s}}\left(\int_{\mathbb{R}}dy{|g_{2}(v(y),w(y))|^{2}}\int_{N}^{\infty}\frac{|x|^{2l}}{(1+\epsilon |x|)^{2s}(1+|x-y|)^{2s}}\right)^{1/2}\nonumber\\
&\leq&2^{s}||\widehat{k}_{11}||_{H^{s}}B^{1/2}\left(\left(\int_{-\infty}^{N}+\int_{N}^{\infty}\right)dy{|g_{1}(v(y),w(y))|^{2}}\frac{|y|^{2l}}{(1+\epsilon |y|)^{2s}}\right)^{1/2}\nonumber\\
&&+2^{s}||\widehat{k}_{12}||_{H^{s}}B^{1/2}\left(\left(\int_{-\infty}^{N}+\int_{N}^{\infty}\right)dy{|g_{2}(v(y),w(y))|^{2}}\frac{|y|^{2l}}{(1+\epsilon |y|)^{2s}}\right)^{1/2}\nonumber\\
&&+2^{s}||\widehat{k}_{11}||_{H^{s}}\left(\int_{-N}^{N}dy{|g_{1}(v(y),w(y))|^{2}}\int_{N}^{\infty}\frac{x^{2l}dx}{(1+\epsilon x)^{2s}(1+|y-x|)^{2s}}\right)^{1/2}\nonumber\\
&&+2^{s}||\widehat{k}_{12}||_{H^{s}}\left(\int_{-N}^{N}dy{|g_{2}(v(y),w(y))|^{2}}\int_{N}^{\infty}\frac{x^{2l}dx}{(1+\epsilon x)^{2s}(1+|y-x|)^{2s}}\right)^{1/2}\nonumber\\
&\leq&2^{s}\delta B^{1/2}||\widehat{k}_{11}||_{H^{s}}\left(\left(\int_{-\infty}^{N}+\int_{N}^{\infty}\right)dy{(v(y)^{2}+w(y)^{2})}\frac{|y|^{2l}}{(1+\epsilon |y|)^{2s}}\right)^{1/2}\nonumber\\
&\leq&2^{s}\delta B^{1/2}||\widehat{k}_{12}||_{H^{s}}\left(\left(\int_{-\infty}^{N}+\int_{N}^{\infty}\right)dy{(v(y)^{2}+w(y)^{2})}\frac{|y|^{2l}}{(1+\epsilon |y|)^{2s}}\right)^{1/2}\nonumber\\
&&+2^{s}||\widehat{k}_{11}||_{H^{s}}\left(\int_{-N}^{N}dy{|g_{1}(v(y),w(y))|^{2}}\int_{-\infty}^{\infty}\frac{(|x|+N)^{2l}dx}{(1+|x|)^{2s}}\right)^{1/2}\nonumber\\
&&+2^{s}||\widehat{k}_{12}||_{H^{s}}\left(\int_{-N}^{N}dy{|g_{2}(v(y),w(y))|^{2}}\int_{-\infty}^{\infty}\frac{(|x|+N)^{2l}dx}{(1+|x|)^{2s}}\right)^{1/2}\nonumber\\
&=&2^{s}\delta B^{1/2}||\widehat{k}_{11}||_{H^{s}}\left(\left(\int_{-\infty}^{N}+\int_{N}^{\infty}\right)dy{(v_{\epsilon}(y)^{2}+w_{\epsilon}(y)^{2})}\right)^{1/2}\nonumber\\
&&+2^{s}\delta B^{1/2}||\widehat{k}_{12}||_{H^{s}}\left(\left(\int_{-\infty}^{N}+\int_{N}^{\infty}\right)dy{(v_{\epsilon}(y)^{2}+w_{\epsilon}(y)^{2})}\right)^{1/2}\nonumber\\
&&+2^{s}R||\widehat{k}_{11}||_{H^{s}}\left(\int_{-N}^{N}dy{|g_{1}(v(y),w(y))|^{2}}\right)^{1/2}\nonumber\\
&&+2^{s}R||\widehat{k}_{12}||_{H^{s}}\left(\int_{-N}^{N}dy{|g_{2}(v(y),w(y))|^{2}}\right)^{1/2}.\label{fnlsC13a}
\end{eqnarray}
Similarly
\begin{eqnarray}
\left(\int_{-\infty}^{N}|v_{\epsilon}(x)|^{2}dx\right)^{1/2}&\leq&
2^{s}\delta B^{1/2}||\widehat{k}_{11}||_{H^{s}}\left(\left(\int_{-\infty}^{N}+\int_{N}^{\infty}\right)dy{(v_{\epsilon}(y)^{2}+w_{\epsilon}(y)^{2})}\right)^{1/2}\nonumber\\
&&+2^{s}\delta B^{1/2}||\widehat{k}_{12}||_{H^{s}}\left(\left(\int_{-\infty}^{N}+\int_{N}^{\infty}\right)dy{(v_{\epsilon}(y)^{2}+w_{\epsilon}(y)^{2})}\right)^{1/2}\nonumber\\
&&+2^{s}R||\widehat{k}_{11}||_{H^{s}}\left(\int_{-N}^{N}dy{|g_{1}(v(y),w(y))|^{2}}\right)^{1/2}\nonumber\\
&&+2^{s}R||\widehat{k}_{12}||_{H^{s}}\left(\int_{-N}^{N}dy{|g_{2}(v(y),w(y))|^{2}}\right)^{1/2}.\label{fnlsC13b}
\end{eqnarray}
Adding (\ref{fnlsC13a}) and (\ref{fnlsC13b}) leads to
\begin{eqnarray*}
&&\left(\int_{N}^{\infty}|v_{\epsilon}(x)|^{2}dx\right)^{1/2}+\left(\int_{-\infty}^{N}|v_{\epsilon}(x)|^{2}dx\right)^{1/2}\nonumber\\
&\leq &
2^{s+1}\delta B^{1/2}||\widehat{k}_{11}||_{H^{s}}\left(\left(\int_{-\infty}^{N}+\int_{N}^{\infty}\right)dy{(v_{\epsilon}(y)^{2}+w_{\epsilon}(y)^{2})}\right)^{1/2}\nonumber\\
&&+2^{s+1}\delta B^{1/2}||\widehat{k}_{12}||_{H^{s}}\left(\left(\int_{-\infty}^{N}+\int_{N}^{\infty}\right)dy{(v_{\epsilon}(y)^{2}+w_{\epsilon}(y)^{2})}\right)^{1/2}\nonumber\\
&&+2^{s+1}R\left(||\widehat{k}_{11}||_{H^{s}}\left(\int_{-N}^{N}dy{|g_{1}(v(y),w(y))|^{2}}\right)^{1/2}\right.\\
&&\left.
+||\widehat{k}_{12}||_{H^{s}}\left(\int_{-N}^{N}dy{|g_{2}(v(y),w(y))|^{2}}\right)^{1/2}\right).
\end{eqnarray*}
The same arguments can be used to obtain similar estimates for $w_{\epsilon}$, as
\begin{eqnarray*}
&&\left(\int_{N}^{\infty}|w_{\epsilon}(x)|^{2}dx\right)^{1/2}+\left(\int_{-\infty}^{N}|w_{\epsilon}(x)|^{2}dx\right)^{1/2}\\
&\leq &
2^{s+1}\delta B^{1/2}||\widehat{k}_{21}||_{H^{s}}\left(\left(\int_{-\infty}^{N}+\int_{N}^{\infty}\right)dy{(v_{\epsilon}(y)^{2}+w_{\epsilon}(y)^{2})}\right)^{1/2}\nonumber\\
&&+2^{s+1}\delta B^{1/2}||\widehat{k}_{22}||_{H^{s}}\left(\left(\int_{-\infty}^{N}+\int_{N}^{\infty}\right)dy{(v_{\epsilon}(y)^{2}+w_{\epsilon}(y)^{2})}\right)^{1/2}\nonumber\\
&&+2^{s+1}R\left(||\widehat{k}_{21}||_{H^{s}}\left(\int_{-N}^{N}dy{|g_{1}(v(y),w(y))|^{2}}\right)^{1/2}\right.\\
&&\left.
+||\widehat{k}_{22}||_{H^{s}}\left(\int_{-N}^{N}dy{|g_{2}(v(y),w(y))|^{2}}\right)^{1/2}\right).
\end{eqnarray*}
Now we take
\begin{eqnarray}
\delta<\frac{1}{2^{s+1}B^{1/2}\sum_{i,j}||\widehat{k}_{ij}||_{H^{s}}},\label{fnlsC14a1}
\end{eqnarray}
to have
\begin{eqnarray}
&&\left(\int_{-\infty}^{N}|v_{\epsilon}(x)|^{2}dx\right)^{1/2}+\left(\int_{N}^{\infty}|v_{\epsilon}(x)|^{2}dx\right)^{1/2}+\left(\int_{-\infty}^{N}|w_{\epsilon}(x)|^{2}dx\right)^{1/2}+\left(\int_{N}^{\infty}|w_{\epsilon}(x)|^{2}dx\right)^{1/2}\nonumber\\
&\leq &\frac{2^{s+1}R}{1-2^{s+1}B^{1/2}\sum_{i,j}||\widehat{k}_{ij}||_{H^{s}}}\left(\left(||\widehat{k}_{11}||_{H^{s}}+||\widehat{k}_{21}||_{H^{s}}\right)\left(\int_{-N}^{N}dy{|g_{1}(v(y),w(y))|^{2}}\right)^{1/2}\right.\nonumber\\
&&\left.+
\left(||\widehat{k}_{12}||_{H^{s}}+||\widehat{k}_{22}||_{H^{s}}\right)\left(\int_{-N}^{N}dy{|g_{2}(v(y),w(y))|^{2}}\right)^{1/2}\right).\label{fnlsC15a}
\end{eqnarray}
We take $\epsilon\rightarrow 0$ in (\ref{fnlsC15a}) and apply Fatou's lemma to obtain
\begin{eqnarray}
&&\left(\int_{-\infty}^{N}|x|^{2l}|v(x)|^{2}dx\right)^{1/2}+\left(\int_{N}^{\infty}|x|^{2l}|v(x)|^{2}dx\right)^{1/2}\nonumber\\
&&+\left(\int_{-\infty}^{N}|x|^{2l}|w(x)|^{2}dx\right)^{1/2}+\left(\int_{N}^{\infty}|x|^{2l}|w(x)|^{2}dx\right)^{1/2}\nonumber\\
&\leq &\frac{2^{s+1}R}{1-2^{s+1}B^{1/2}\sum_{i,j}||\widehat{k}_{ij}||_{H^{s}}}\left(\left(||\widehat{k}_{11}||_{H^{s}}+||\widehat{k}_{21}||_{H^{s}}\right)\left(\int_{-N}^{N}dy{|g_{1}(v(y),w(y))|^{2}}\right)^{1/2}\right.\nonumber\\
&&\left.+
\left(||\widehat{k}_{12}||_{H^{s}}+||\widehat{k}_{22}||_{H^{s}}\right)\left(\int_{-N}^{N}dy{|g_{2}(v(y),w(y))|^{2}}\right)^{1/2}\right).\label{fnlsC15b}
\end{eqnarray}
Then (\ref{fnlsC8}), the hypothesis $(v,w)\in L^{\infty}\times L^{\infty}$, and (\ref{fnlsC15b}) imply that $(1+|x|^{l})v(x), (1+|x|^{l})w(x)\in L^{2}$, that is $\widehat{v},\widehat{w}\in H^{l}$. In particular, $v,w\in L^{2}$. Since $x\mapsto 1/(1+\epsilon |x|)^{s}$ is also in $L^{2}$ for any $\epsilon>0$, then 
\begin{eqnarray*}
\frac{v(x)}{(1+\epsilon |x|)^{s}}, \frac{w(x)}{(1+\epsilon |x|)^{s}}\in L^{1}.
\end{eqnarray*}
On the other hand, we apply Fubini's theorem, Schwarz inequality, (\ref{fnlsC10}), (\ref{fnlsC11}), the inequality
\begin{eqnarray*}
(v^{2}+w^{2})^{1/2}\leq |v|+|w|,
\end{eqnarray*}
(\ref{fnlsC4}), and Lemma 3.1.1 of \cite{BonaL} to have
\begin{eqnarray*}
&&\int_{N}^{\infty}\frac{v(x)}{(1+\epsilon |x|)^{s}}dx\\
&\leq&\int_{\mathbb{R}}dy|g_{1}(v(y),w(y))|\int_{N}^{\infty}\frac{|k_{11}(x-y)|}{(1+\epsilon |x|)^{s}}dx+\int_{\mathbb{R}}dy|g_{2}(v(y),w(y))|\int_{N}^{\infty}\frac{|k_{12}(x-y)|}{(1+\epsilon |x|)^{s}}dx\nonumber\\
&\leq&\int_{\mathbb{R}}dy|g_{1}(v(y),w(y))|\left(\int_{N}^{\infty}(1+|x-y|)^{2s}|k_{11}(x-y)|^{2}dx\right)^{1/2}\left(\int_{N}^{\infty}\frac{dx}{(1+|x-y|)^{2s}(1+\epsilon |x|)^{2s}}\right)^{1/2}\nonumber\\
&&+\int_{\mathbb{R}}dy|g_{2}(v(y),w(y))|\left(\int_{N}^{\infty}(1+|x-y|)^{2s}|k_{12}(x-y)|^{2}dx\right)^{1/2}\left(\int_{N}^{\infty}\frac{dx}{(1+|x-y|)^{2s}(1+\epsilon |x|)^{2s}}dx\right)^{1/2}\nonumber\\
&\leq&2^{s}||\widehat{k}_{11}||_{H^{s}}B^{1/2}\int_{\mathbb{R}}\frac{|g_{1}(v(y),w(y))|dy}{(1+\epsilon |y|)^{s}}+2^{s}||\widehat{k}_{12}||_{H^{s}}B^{1/2}\int_{\mathbb{R}}\frac{|g_{2}(v(y),w(y))|dy}{(1+\epsilon |y|)^{s}}\nonumber\\
&\leq&2^{s}||\widehat{k}_{11}||_{H^{s}}\delta B^{1/2}\left(\int_{-\infty}^{-N}+\int_{N}^{\infty}\right)\frac{(v(y)^{2}+w(y)^{2})^{1/2}dy}{(1+\epsilon |y|)^{s}}\nonumber\\
&&+2^{s}||\widehat{k}_{12}||_{H^{s}}\delta B^{1/2}\left(\int_{-\infty}^{-N}+\int_{N}^{\infty}\right)\frac{(v(y)^{2}+w(y)^{2})^{1/2}dy}{(1+\epsilon |y|)^{s}}\nonumber\\
&&+2^{s}||\widehat{k}_{11}||_{H^{s}} B^{1/2}\int_{-N}^{N}{|g_{1}(v(y),w(y))|dy}+2^{s}||\widehat{k}_{12}||_{H^{s}} B^{1/2}\int_{-N}^{N}{|g_{2}(v(y),w(y))|dy}\nonumber\\
&\leq&2^{s}||\widehat{k}_{11}||_{H^{s}}\delta B^{1/2}\left(\int_{-\infty}^{-N}+\int_{N}^{\infty}\right)\frac{(|v(y)|+|w(y)|)dy}{(1+\epsilon |y|)^{s}}\nonumber\\
&&+2^{s}||\widehat{k}_{12}||_{H^{s}}\delta B^{1/2}\left(\int_{-\infty}^{-N}+\int_{N}^{\infty}\right)\frac{(|v(y)|+|w(y)|)dy}{(1+\epsilon |y|)^{s}}\nonumber\\
&&+2^{s}||\widehat{k}_{11}||_{H^{s}} B^{1/2}\int_{-N}^{N}{|g_{1}(v(y),w(y))|dy}+2^{s}||\widehat{k}_{12}||_{H^{s}} B^{1/2}\int_{-N}^{N}{|g_{2}(v(y),w(y))|dy}.
\end{eqnarray*}
Similarly
\begin{eqnarray*}
\int_{-\infty}^{-N}\frac{v(x)}{(1+\epsilon |x|)^{s}}dx&\leq&2^{s}||\widehat{k}_{11}||_{H^{s}}\delta B^{1/2}\left(\int_{-\infty}^{-N}+\int_{N}^{\infty}\right)\frac{(|v(y)|+|w(y)|)dy}{(1+\epsilon |y|)^{s}}\\
&&+2^{s}||\widehat{k}_{12}||_{H^{s}}\delta B^{1/2}\left(\int_{-\infty}^{-N}+\int_{N}^{\infty}\right)\frac{(|v(y)|+|w(y)|)dy}{(1+\epsilon |y|)^{s}}\nonumber\\
&&+2^{s}||\widehat{k}_{11}||_{H^{s}} B^{1/2}\int_{-N}^{N}{|g_{1}(v(y),w(y))|dy}+2^{s}||\widehat{k}_{12}||_{H^{s}} B^{1/2}\int_{-N}^{N}{|g_{2}(v(y),w(y))|dy},
\end{eqnarray*}
and
\begin{eqnarray*}
\int_{N}^{\infty}\frac{w(x)}{(1+\epsilon |x|)^{s}}dx&\leq&
2^{s}||\widehat{k}_{21}||_{H^{s}}\delta B^{1/2}\left(\int_{-\infty}^{-N}+\int_{N}^{\infty}\right)\frac{(|v(y)|+|w(y)|)dy}{(1+\epsilon |y|)^{s}}\\
&&+2^{s}||\widehat{k}_{22}||_{H^{s}}\delta B^{1/2}\left(\int_{-\infty}^{-N}+\int_{N}^{\infty}\right)\frac{(|v(y)|+|w(y)|)dy}{(1+\epsilon |y|)^{s}}\nonumber\\
&&+2^{s}||\widehat{k}_{21}||_{H^{s}} B^{1/2}\int_{-N}^{N}{|g_{1}(v(y),w(y))|dy}+2^{s}||\widehat{k}_{22}||_{H^{s}} B^{1/2}\int_{-N}^{N}{|g_{2}(v(y),w(y))|dy},
\end{eqnarray*}
\begin{eqnarray*}
\int_{-\infty}^{-N}\frac{w(x)}{(1+\epsilon |x|)^{s}}dx&\leq&
2^{s}||\widehat{k}_{21}||_{H^{s}}\delta B^{1/2}\left(\int_{-\infty}^{-N}+\int_{N}^{\infty}\right)\frac{(|v(y)|+|w(y)|)dy}{(1+\epsilon |y|)^{s}}\\
&&+2^{s}||\widehat{k}_{22}||_{H^{s}}\delta B^{1/2}\left(\int_{-\infty}^{-N}+\int_{N}^{\infty}\right)\frac{(|v(y)|+|w(y)|)dy}{(1+\epsilon |y|)^{s}}\nonumber\\
&&+2^{s}||\widehat{k}_{21}||_{H^{s}} B^{1/2}\int_{-N}^{N}{|g_{1}(v(y),w(y))|dy}+2^{s}||\widehat{k}_{22}||_{H^{s}} B^{1/2}\int_{-N}^{N}{|g_{2}(v(y),w(y))|dy}.
\end{eqnarray*}
Therefore, using (\ref{fnlsC14a1}), it holds that
\begin{eqnarray*}
\left(\int_{-\infty}^{-N}+\int_{N}^{\infty}\right)\frac{(|v(y)|+|w(y)|)dy}{(1+\epsilon |y|)^{s}}&\leq&
\frac{2^{s+1}B^{1/2}}{1-2^{s+1}B^{1/2}\delta\sum_{i,j}||\widehat{k}_{ij}||_{H^{s}}}\left(\left(||\widehat{k}_{11}||_{H^{s}}\right.\right.\\
&&\left.\left.+||\widehat{k}_{21}||_{H^{s}}\right)\int_{-N}^{N}dy|g_{1}(v(y),w(y))|\right.\nonumber\\
&&+\left. \left(||\widehat{k}_{12}||_{H^{s}}+||\widehat{k}_{22}||_{H^{s}}\right)\int_{-N}^{N}dy{|g_{2}(v(y),w(y))|}\right).
\end{eqnarray*}
Taking $\epsilon\rightarrow 0$ and using Fatou's lemma yield
\begin{eqnarray*}
\left(\int_{-\infty}^{-N}+\int_{N}^{\infty}\right){(|v(y)|+|w(y)|)dy}&\leq&
\frac{2^{s+1}B^{1/2}\sum_{i,j}||\widehat{k}_{ij}||_{H^{s}}}{1-2^{s+1}B^{1/2}\delta\sum_{i,j}||\widehat{k}_{ij}||_{H^{s}}}\left(\int_{-N}^{N}dy{|g_{1}(v(y),w(y))|}\right.\\
&&\left.+\int_{-N}^{N}dy{|g_{2}(v(y),w(y))|}\right),
\end{eqnarray*}
which implies $v,w\in L^{1}$.

We now multiply (\ref{fnlsC4}) by $|x|^{l}$ and apply the property
\begin{eqnarray*}
|x|^{l}\leq \beta(|x-y|^{l}+|y|^{l}),\quad \beta=\max\{1,2^{l-1}\},
\end{eqnarray*}
to obtain
\begin{eqnarray}
|x|^{l}|v(x)|&\leq&\beta\int_{\mathbb{R}}|x-y|^{l}|k_{11}(x-y)g_{1}(v(y),w(y))|dy+\beta\int_{\mathbb{R}}|y|^{l}|k_{11}(x-y)g_{1}(v(y),w(y))|dy\nonumber\\
&&+\beta\int_{\mathbb{R}}|x-y|^{l}|k_{12}(x-y)g_{2}(v(y),w(y))|dy\nonumber\\
&&+\beta\int_{\mathbb{R}}|y|^{l}|k_{12}(x-y)g_{2}(v(y),w(y))|dy,\label{fnlsC17}
\end{eqnarray}
and a similar inequality for $|x|^{l}|v(x)|$. Since $k_{ij}, |x|^{l}k_{ij}, g_{i}(v,w), |x|^{l}g_{i}(v,w)$ are in $L^{2}$, then the right-hand side of (\ref{fnlsC17}) is the convolution of $L^{2}$ functions and, therefore, it is a bounded function. This implies 
\begin{eqnarray*}
|x|^{l}v(x), |x|^{l}w(x)\in L^{\infty},
\end{eqnarray*}
and completes the proof.
\end{proof}
Now, the arguments used in Corollary 3.1.3 of \cite{BonaL} can be applied to each kernel $k_{ij}, i,j=1,2$, nonlinear terms $g_{i}, i=1,2$ and the system for $v,w$, leading to
\begin{corollary}
\label{corol1}
Let $(v,w)$ be a solution of (\ref{fnlsC4}) satisfying the assumptions of Theorem \ref{theor1}. Then $\widehat{v},\widehat{w}\in H^{s}$ and
\begin{eqnarray*}
|x|^{\alpha}v(x), |x|^{\alpha}w(x)\in L^{\infty},
\end{eqnarray*}
for any $0<\alpha\leq s$.
\end{corollary}
The following result establishes the asymptotic decay of solutions of (\ref{fnlsC4}); the proof relies on an adaptation of Theorem 3.1.5 of \cite{BonaL}
\begin{theorem}
\label{theor2} Let $(v,w)$ be a solution of (\ref{fnlsC4}) satisfying the assumptions of Theorem \ref{theor1}, and $m=2s+1$. Then
\begin{eqnarray}
\lim_{|x|\rightarrow\infty}|x|^{2s+1}v(x)&=&K_{1}\int_{\mathbb{R}}g_{1}(v(y),w(y))dy,\nonumber\\
\lim_{|x|\rightarrow\infty}|x|^{2s+1}w(x)&=&K_{1}\int_{\mathbb{R}}g_{2}(v(y),w(y))dy,\label{fnlsC18}
\end{eqnarray}
where $K_{1}$ is given in (\ref{fnlsC7a}).
\end{theorem}
\begin{proof}
Since the proof of (\ref{fnlsC18}) is similar for both equalities, we focus on the first one. 
We first show that for any $l$ with $0<l<m$
\begin{eqnarray}
\lim_{|x|\rightarrow\infty}|x|^{l}v(x)=0.\label{fnlsC19a}
\end{eqnarray}
We take $l_{1}>0$ such that $m-1/2<l_{1}<\min\{m, r(m-1/2)\}$ where $r=2\sigma+1$. From (\ref{fnlsC7a}), (\ref{fnlsC7b}), and Corollary \ref{corol1}, we have, for $i,j=1,2$
\begin{eqnarray}
&&\lim_{|x|\rightarrow\infty}|x|^{l_{1}}k_{ij}(x)=0\nonumber\\
&&|x|^{l_{1}}|g_{j}(v,w)(x)|\leq (|u(x)|_{2}|x|^{\frac{l_{1}}{r}})^{r}\rightarrow 0,\quad |x|\rightarrow\infty.\label{fnlsC19b}
\end{eqnarray}
Furthermore, since $k_{ij}, g_{i}\in L^{1}\cap L^{2}$, we have
\begin{eqnarray*}
|x|^{l_{1}}|v(x)|&\leq&\beta\int_{\mathbb{R}}|x-y|^{l_{1}}|k_{11}(x-y)g_{1}(v(y),w(y))|dy\\
&&+\beta\int_{\mathbb{R}}|y|^{l_{1}}|k_{11}(x-y)g_{1}(v(y),w(y))|dy\nonumber\\
&&+\beta\int_{\mathbb{R}}|x-y|^{l_{1}}|k_{12}(x-y)g_{2}(v(y),w(y))|dy\\
&&+\beta\int_{\mathbb{R}}|y|^{l_{1}}|k_{12}(x-y)g_{2}(v(y),w(y))|dy,
\end{eqnarray*}
where $\beta=\max\{1,2^{l_{1}-1}\}$. Therefore, using (\ref{fnlsC19b}) leads to $|x|^{l_{1}}|v(x)|\rightarrow 0$ as $|x|\rightarrow \infty$. Similar arguments prove that $|x|^{l_{1}}|w(x)|\rightarrow 0$ as $|x|\rightarrow \infty$.  If $l_{1}<m$, we choose $l_{2}\in (l_{1}, \min\{m, rl_{1}\})$ and repeat the argument to have
\begin{eqnarray*}
|x|^{l_{2}}|v(x)|\rightarrow 0,\quad |x|^{l_{2}}|w(x)|\rightarrow 0,
\end{eqnarray*} 
as $|x|\rightarrow \infty$. Iterating the argument a finite number of times, cf. \cite{BonaL}, leads to (\ref{fnlsC19a}) and
similarly
\begin{eqnarray}
\lim_{|x|\rightarrow\infty}|x|^{l}w(x)=0.\label{fnlsC19aa}
\end{eqnarray}
Now, since (\ref{fnlsC7a}), (\ref{fnlsC7b}) hold, then there are constants $N_{0}, A>0$ such that
\begin{eqnarray*}
|x|^{m}|k_{ij}(x)|\leq A,\quad |k_{ij}(x)|\leq A,\quad |x|\geq N_{0}.
\end{eqnarray*}
In addition, under hypotheses of theorem \ref{theor1}, for any $\epsilon>0$ there is $N\geq N_{0}$ such that, for $i,j=1,2$
\begin{eqnarray*}
&&\int_{N}^{\infty}|g_{i}(v(x),w(x))|dx<\epsilon,\quad \int_{N}^{\infty}|g_{i}(v(x),w(x))|^{2}dx<\epsilon^{2},\quad \int_{-\infty}^{N}|g_{i}(v(x),w(x))|dx<\epsilon,\\
&&\int_{N}^{\infty}|k_{ij}(x)|dx<\epsilon,\quad \int_{-\infty}^{N}|k_{ij}(x)|dx<\epsilon,\quad |y|^{m}|g_{i}(v(y),w(y))|< \epsilon, \;\;|y|\geq N.
\end{eqnarray*}
We may also assume that, for any $t\in (-N,N), x\geq N+N_{0}$
\begin{eqnarray*}
|x^{m}k_{ij}(x-y)-\widetilde{K}|<\epsilon,\quad \widetilde{K}=\left\{\begin{matrix}K_{1}&i=j\\0&i\neq j.\end{matrix}\right.
\end{eqnarray*}
Then, for $x$ large enough
\begin{eqnarray*}
|x^{m}v(x)-K_{1}\int_{\mathbb{R}}g_{1}(v(y),w(y))dy|&\leq& \left|\int_{|y|\leq N}(x^{m}k_{11}(x-y)-K_{1})g_{1}(v(y),w(y))dy\right|\\
&&+ \left|\int_{|y|\leq N}(x^{m}k_{12}(x-y))g_{2}(v(y),w(y))dy\right|\\
&&+\left|K_{1}\int_{|y|\geq N}g_{1}(v(y),w(y))dy\right|\\
&&+2^{m-1}\left(\int_{-\infty}^{-N}|x-y|^{m}|k_{11}(x-y)g_{1}(v(y),w(y))|dy\right.\\
&&+\left.\int_{-\infty}^{-N}|y|^{m}|k_{11}(x-y)g_{1}(v(y),w(y))|dy\right.\\
&&\left.+\int_{-\infty}^{-N}|x-y|^{m}|k_{12}(x-y)g_{2}(v(y),w(y))|dy\right.\\
&&\left.+\int_{-\infty}^{-N}|y|^{m}|k_{12}(x-y)g_{2}(v(y),w(y))|dy\right)\\
&&+2^{m-1}\left(\int_{N,N_{0}}|x-y|^{m}|k_{11}(x-y)g_{1}(v(y),w(y))|dy\right.\\
&&\left.+\int_{N,N_{0}}|x-y|^{m}|k_{12}(x-y)g_{2}(v(y),w(y))|dy\right)\\
&&+2^{m-1}\left(\int_{N}^{\infty}|y|^{m}|k_{11}(x-y)g_{1}(v(y),w(y))|dy\right.\\
&&\left.+\int_{N}^{\infty}|y|^{m}|k_{12}(x-y)g_{2}(v(y),w(y))|dy\right)\\
&\leq&K_{0}\epsilon,
\end{eqnarray*}
where
\begin{eqnarray*}
\int_{N,N_{0}}&:=&\int_{N}^{x-N_{0}}+\int_{x-N_{0}}^{x+N_{0}}+\int_{x+N_{0}}^{\infty},\\
K_{0}&:=&||g_{1}(v,w)||_{L^{1}}+||g_{2}(v,w)||_{L^{1}}+2|K_{1}|+2^{m}3A\\
&&+2^{m-1}N_{0}^{m}(||k_{1}||_{L^{2}}+||k_{12}||_{L^{2}}+2^{m}(||k_{1}||_{L^{1}}+||k_{12}||_{L^{1}}.
\end{eqnarray*} 
Since $\epsilon$ is arbitrary, then the first limit in  (\ref{fnlsC18}) holds. Similar arguments can be applied to
\begin{eqnarray*}
|x^{m}v(x)-K_{1}\int_{\mathbb{R}}g_{1}(v(y),w(y))dy|,
\end{eqnarray*}
to prove the second one.
\end{proof}
\section{Numerical experiments}
\label{sec4}
This section is devoted to the numerical illustration of theorem \ref{theor2}. To this end, we make use of the numerical procedure introduced in \cite{DR1} to generate approximate solitary wave profiles. We recall that the scheme consists of writing (\ref{fnls22_2}) in a fixed-point form
\begin{eqnarray*}
Q\begin{pmatrix}v\\w\end{pmatrix}=G(v,w)=(v^{2}+w^{2})^{\sigma+1}\begin{pmatrix}v\\w\end{pmatrix},\label{fnls_2327}
\end{eqnarray*}
where $Q$ and $G$ are defined in (\ref{fnls_235}), (\ref{fnlsC3b}), respectively, and solving it iteratively by the Petviashvili method, \cite{Petv1976}
\begin{eqnarray}
m_{\nu}&=&\frac{\langle Qz^{[\nu]},z^{[\nu]}\rangle}{\langle G(z^{[\nu]}),z^{[\nu]}\rangle},\nonumber\\
Qz^{[\nu+1]}&=&m_{\nu}^{\alpha}G(z^{[\nu]}),\quad \nu=0,1,\ldots,\label{fnls_311}
\end{eqnarray}
where $\langle\cdot,\cdot\rangle$ is the $L^{2}$-inner product and $\alpha\in (1,(2\sigma+2)/2\sigma)$, with $\alpha=(2\sigma+1)/2\sigma$ as optimal choice, \cite{pelinovskys}. The implementation makes use of a Fourier collocation approximation of the profiles on a long enough interval $(-L,L)$ and vector extrapolation techniques to accelerate the convergence of (\ref{fnls_311}), \cite{sidi,sidifs,smithfs,AlvarezD2016}.

The experiments illustrating and specifying the asymptotic decay of the solitary waves are displayed in Figures \ref{SPAM_Fig1}-\ref{SPAM_Fig3}, for approximate profiles generated with the values  $\lambda_{0}^{1}=1,\sigma=1,\theta(x)=x^{2}$, speed, represented by the second Lagrange multiplier $c_{s}=\lambda_{0}^{2}=0.25$, and
different values of the parameter $s$. They show, in log-log scale, the real component $v$ and the modulus $\rho=\sqrt{v^{2}+w^{2}}$ of the computed profiles, for different values of the length of the interval $(-L,L)$. The slopes of the lines are compared with that of the dashed line, confirming a decay like $1/|x|^{2s+1}$. The dependence on $L$  makes necessarily influence in the accuracy of the method, and must be taken into account when studying numerically the dynamics of the waves, where long time computations are required, \cite{DR2}.
\begin{figure}[htbp]
\centering
\subfigure[]
{\includegraphics[width=6cm]{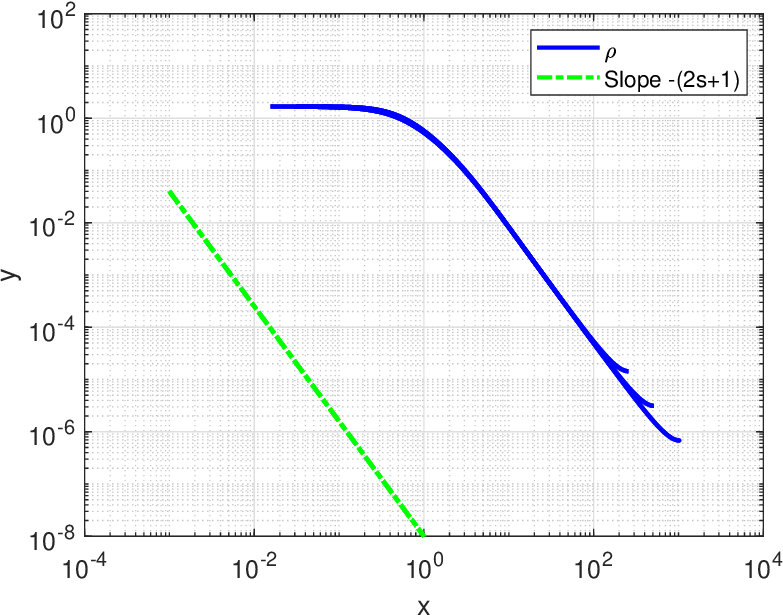}}
\subfigure[]
{\includegraphics[width=6.5cm]{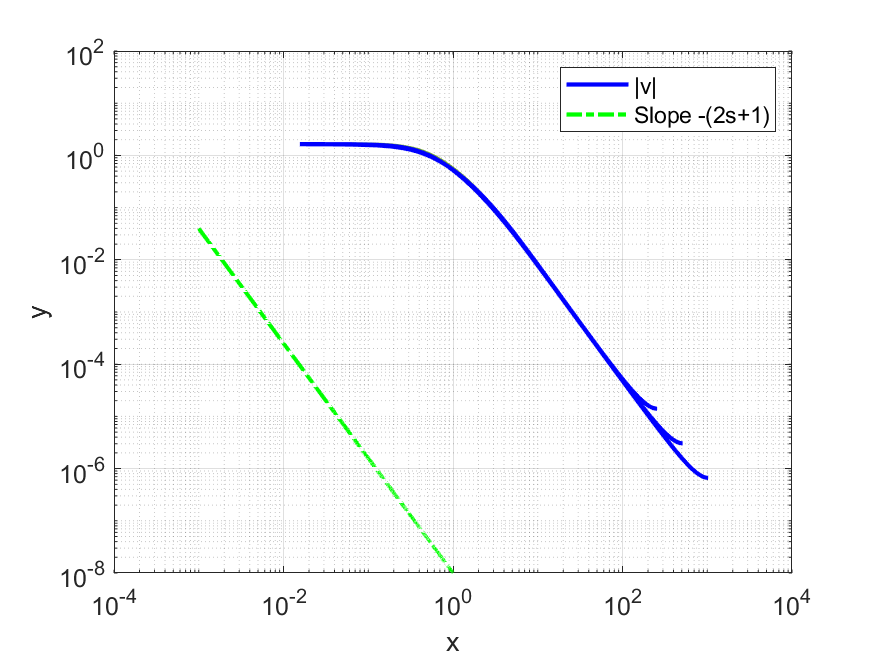}}
\caption{Approximate profiles for $\theta(x)=x^{2}, \sigma=1, s=0.6, \lambda_{0}^{1}=1,\lambda_{0}^{2}=0.25$ in log-log scale, for $L=256,512,1024$. (a) $\rho=\sqrt{v^{2}+w^{2}}$; (b) $v$. The slope of the dashed line is $-(2s+1)$.}
\label{SPAM_Fig1}
\end{figure}
\begin{figure}[htbp]
\centering
\subfigure[]
{\includegraphics[width=6cm]{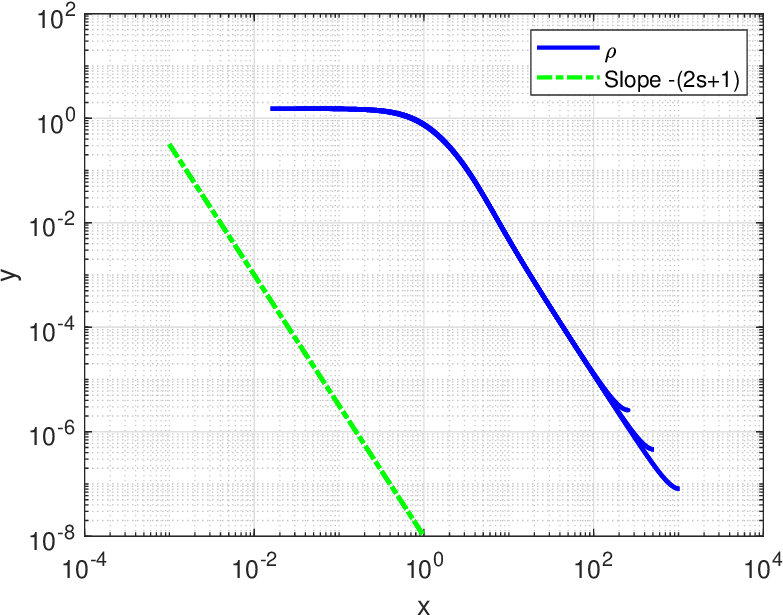}}
\subfigure[]
{\includegraphics[width=6.5cm]{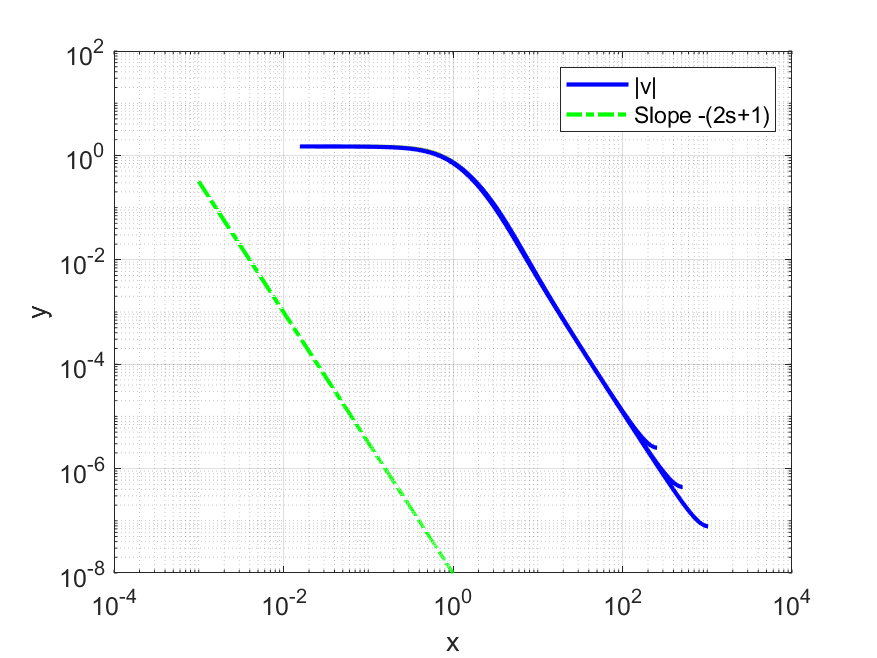}}
\caption{Approximate profiles for $\theta(x)=x^{2},\sigma=1, s=0.75, \lambda_{0}^{1}=1,\lambda_{0}^{2}=0.25$ in log-log scale, for $L=256,512,1024$. (a) $\rho=\sqrt{v^{2}+w^{2}}$; (b) $v$. The slope of the dashed line is $-(2s+1)$.}
\label{SPAM_Fig2}
\end{figure}
\begin{figure}[htbp]
\centering
\subfigure[]
{\includegraphics[width=6cm]{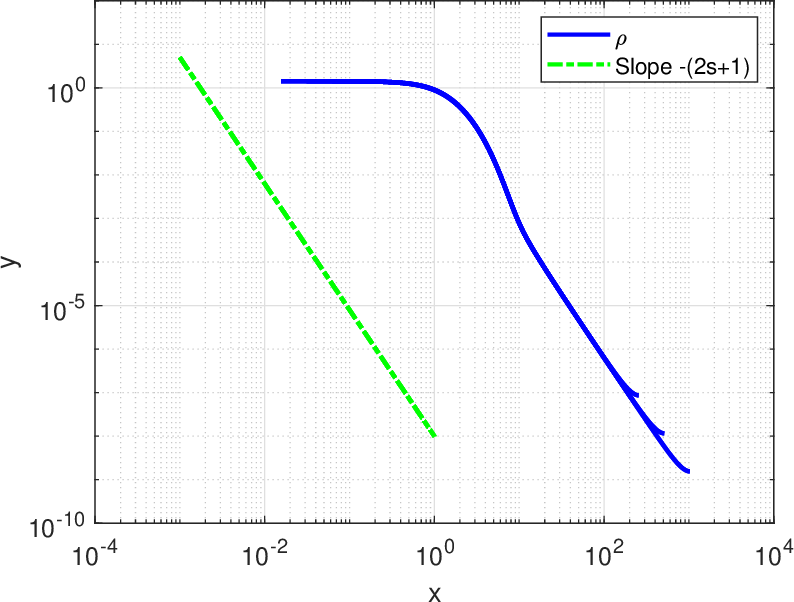}}
\subfigure[]
{\includegraphics[width=6.5cm]{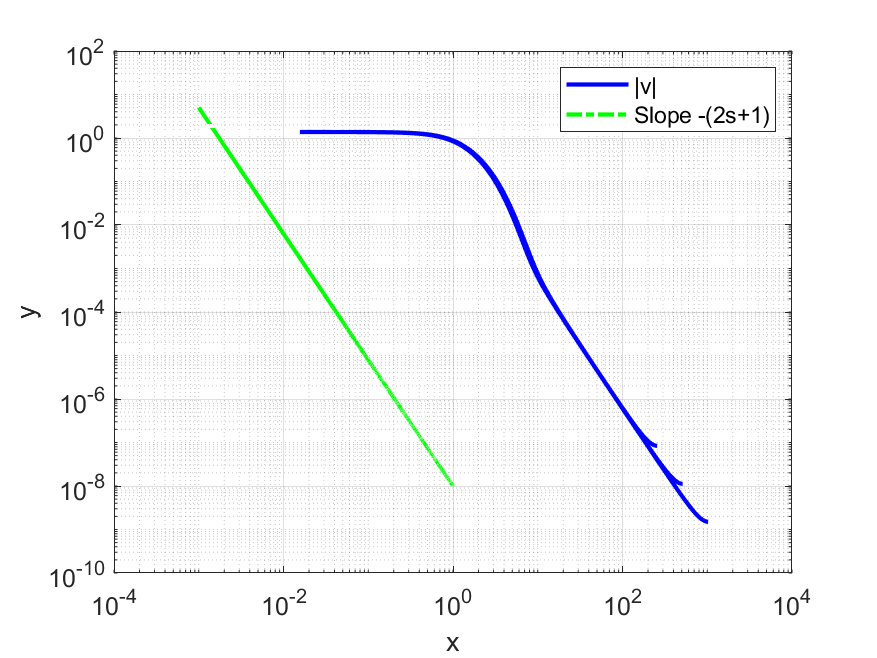}}
\caption{Approximate profiles for $\theta(x)=x^{2},\sigma=1, s=0.95, \lambda_{0}^{1}=1,\lambda_{0}^{2}=0.25$ in log-log scale, for $L=256,512,1024$. (a) $\rho=\sqrt{v^{2}+w^{2}}$; (b) $v$. The slope of the dashed line is $-(2s+1)$.}
\label{SPAM_Fig3}
\end{figure}

%

\section*{Acknowledgments}
This research has been supported by Ministerio de Ciencia e Innovaci\'on project PID2023-147073NB-I00.

%
%
%


\bigskip
%

\end{document}